\documentclass[a4paper]{amsart}
\usepackage{graphicx}
\usepackage{amssymb}
\usepackage[mathscr]{euscript}
\usepackage{tikz}
\usetikzlibrary{matrix,arrows,decorations.pathmorphing,decorations.markings,calc}
\usepackage{tikz-cd}
\usepackage[utf8]{inputenc}
\usepackage{lmodern}
\usepackage[T1]{fontenc}

\theoremstyle{plain}
  \newtheorem{thm}{Theorem}
  \newtheorem{defn}{Definition}
  \newtheorem{prop}{Proposition}
  \newtheorem{cor}{Corollary}
  
\theoremstyle{definition}
  \newtheorem{example}{Example}
  \newtheorem*{rem}{Remark}

\newcommand{\mf}{\mathfrak}
\newcommand{\on}{\operatorname}

\newcommand{\g}{\mathfrak{g}}
\newcommand{\h}{\mathfrak{h}}

\newcommand{\Hom}{\operatorname{Hom}}
\newcommand{\la}{\langle}
\newcommand{\ra}{\rangle}

\newcommand{\half}{\frac{1}{2}}
\newcommand{\cat}{\mathscr}

\newcommand{\Mod}{\text{-}\mathrm{Mod}}

\begin{document}

\title[On quantization of Poisson-Lie groups and moduli spaces]{On deformation quantization of Poisson-Lie groups and moduli spaces of flat connections}
\author{David Li-Bland}
\author{Pavol \v{S}evera}

\address{Department of Mathematics, University of California, Berkeley}
\email{libland@math.berkeley.edu}
\address{Department of Mathematics, Universit\'{e} de Gen\`{e}ve, Geneva, Switzerland}
\email{pavol.severa@gmail.com}
\thanks{D.L-B. was supported by the National Science Foundation under Award No. DMS-1204779.
\\
P.\v S. was supported in part by  the grant MODFLAT of the European Research Council and the NCCR SwissMAP of the Swiss National Science Foundation..}

\begin{abstract}
We give simple explicit formulas for deformation quantization of Poisson-Lie groups and of similar Poisson manifolds which can be represented as moduli spaces of flat connections on surfaces. The star products depend on a choice of Drinfeľd associator and are obtained by applying certain monoidal functors (fusion and reduction) to commutative algebras in Drinfeľd categories. From a geometric point of view this construction can be understood as a quantization of the quasi-Poisson structures on moduli spaces of flat connections.

\end{abstract}
\maketitle

\section{Introduction}

This note is based on the observation made in \cite{LiBland:2012vo} that Poisson-Lie groups, and many related Poisson manifolds, can be realized as moduli spaces of flat connections on surfaces with boundary and marked points, which makes the problem of their deformation quantization rather straightforward.

Alekseev, Kosmann-Schwarzbach and Meinrenken \cite{Alekseev00}  noticed that moduli spaces of flat $\g$-connections on surfaces with boundary and marked points carry a \emph{quasi-Poisson} structure, and that they can be constructed out of simple building blocks using the operation of \emph{fusion}. In \cite{LiBland:2012vo} we extended their result to surfaces with several marked points on each boundary component (which simplified the basic building block to just the Lie group $G$, with zero quasi-Poisson bracket), and also observed how coisotropic reduction can be used to produce various Poisson manifolds, and Poisson-Lie groups in particular.

Enriquez and Etingof \cite{Enriquez:2003tw} realized that the quantization of a $\g$-quasi-Poisson manifold should be an associative algebra in Drinfeľd's braided monoidal category $U\g\Mod^\Phi$, where $\Phi$ is a Drinfeľd associator.

Our strategy is to start with a commutative algebra in a Drinfeľd category and then apply to it a monoidal functor to form an associative algebra in the category of vector spaces. The monoidal functor is a composition of a quantum analogue of the quasi-Poisson operations of fusion (basically an obvious monoidal structure on the tensor product functor in a braided monoidal category) and coisotropic reduction.

We should stress that our aim is rather modest. We just provide a star product for certain Poisson manifolds. We do not discuss how it depends on additional choices (of how we decompose the surface into a union of disks). More importantly, we do not prove that our quantization procedure preserves algebraic structures on the Poisson manifolds; for example, we do not prove that we get Hopf algebras out of Poisson-Lie groups (though the star product is equal (at the origin) to the quantum coproduct of Etingof and Kazhdan  \cite{Etingof:1996bc}). We plan to address these problems, as well as the quantization of quasi-Poisson moment maps (needed for quantization of certain Poisson manifolds), in a future work.

\section{Commutative algebras in the Drinfeľd category}

Let $\g$ be a Lie algebra with a chosen invariant element $t\in(S^2\g)^\g$, and let $\Phi\in\mathbb{C}\la\!\la x,y\ra\!\ra$ be a Drinfeľd associator. The element $t$ and the associator $\Phi$ may be used to deform the symmetric monoidal structure on the category $U\g\Mod$ to a braided monoidal structure.
 More precisely, let $U\g\Mod^\Phi$ be the category with same objects as $U\g\Mod$, and with
$$\Hom_{U\g\Mod^\Phi}(X,Y)=\Hom_{U\g\Mod}(X,Y)[\![\hbar]\!].$$
The tensor products are the same in both categories, but the braiding in $U\g\Mod^\Phi$ is the symmetry in $U\g\Mod$ composed with the action of $\exp(\hbar t^{1,2}/2)$, and the associativity constraint is given by the action of $\Phi(\hbar t^{1,2},\hbar t^{2,3})$. See \cite{Drinfeld:1989tu} for details.

\begin{prop}\label{prop:comm}
Let $A$ be a commutative associative algebra in $U\g\Mod$ with product $m:A\otimes A\to A$, such that 
$$m\circ(t^{1,2}\cdot)=0.$$
Then $A$, with its original product, is a commutative associative algebra in $U\g\Mod^\Phi$.
\end{prop}
\begin{proof}
We need to show that 
$$m\circ(\exp(\hbar t^{1,2}/2)\,\cdot\,)=m\quad \text{and}\quad m^{(3)}\circ(\Phi(\hbar t^{1,2},\hbar t^{2,3})\,\cdot\,)=m^{(3)},$$
 where $m^{(3)}=m\circ(m\otimes 1)=m\circ(1\otimes m)$. Both follow from $m\circ(t^{1,2}\cdot)=0$.
\end{proof}

We shall call a Lie subalgebra $\mf c\subset\g$ \emph{coisotropic} if the image of $t\in S^2\g$ in $S^2(\g/\mf c)$ vanishes.
\begin{cor}
Let $M$ be a manifold with an action of $\g$, such that the stabilizers of points are coisotropic Lie subalgebras of $\g$. Then $C^\infty(M)$, with its original product, is a commutative associative algebra in $U\g\Mod^\Phi$.
\end{cor}

\begin{example}\label{ex:CG}
Let $G$ be a Lie group with the Lie algbera $\g$ and $C\subset G$ a closed Lie subgroup with a coisotropic Lie algebra $\mf c\subset \g$. Then $C^\infty(G/C)$ is a commutative associative algebra in $U\g\Mod^\Phi$.

If $\bar\g$ denotes $\g$ with $t$ replaced by $-t$, then the diagonal Lie subalgebra $\g_\text{diag}\subset\g\oplus\bar\g$ is coisotropic. We have a natural identification $(G\times\bar G)/G_\text{diag}=G$. The algebra $C^\infty(G)$ is thus a commutative associative algebra in $U(\g\oplus\bar\g)\Mod^\Phi$. The action is given explicitly by 
$$(\xi,\eta)\cdot f= (\eta^L-\xi^R)f,\quad f\in C^\infty(G), \quad(\xi,\eta)\in \g\oplus\bar\g,$$
 where $\xi^L/\eta^R$ denote the left/right-invariant vector fields on $G$ equal to $\xi/\eta$ at the identity.
\end{example}

\begin{prop}
Let $\g$ and $\h$ be Lie algebras with chosen elements $t_\g\in(S^2\g)^\g$, $t_\h\in(S^2\h)^\h$. Let $\mf c\subset\g$ be a coisotropic Lie subalgebra. The functor
$$\mf c\text{-invariants}:U(\g\oplus\h)\Mod^\Phi\to U\h\Mod^\Phi$$
is braided monoidal with the coherence maps $X^{\mf c}\otimes Y^{\mf c}\to (X\otimes Y)^{\mf c}$ being the natural inclusion. 
\end{prop}
\begin{proof}
On tensor products of $\mf c$-invariants the element $t_\g$ acts by $0$.
\end{proof}
\begin{cor}
If $A$ is an associative algebra in $U(\g\oplus\h)\Mod^\Phi$ then $A^{\mf c}$, with the product inherited from $A$, is an associative algebra in $ U\h\Mod^\Phi$.
\end{cor}
We shall call the algebra $A^{\mf c}$ the \emph{reduction of $A$ by $\mf{c}$}.
\begin{example}
If we take $A=C^\infty(G)$ in $U(\g\oplus\bar\g)\Mod^\Phi$ and take the $\mf c$-invariants for a coisotropic $\mf c\subset\bar\g$, we get $A^{\mf c}=C^{\infty}(G/C)$, an algebra in $U\g\Mod^\Phi$.
\end{example}

\section{Fusion}\label{sect:fusion}
We shall produce non-commutative algebras using monoidal functors which are not braided monoidal, and applying them to (possibly commutative) associative algebras.

\begin{thm}\label{thm:qfusion}
Let $\cat C$ be a braided monoidal category. The functor
$$\otimes:\cat C\times\cat C\to\cat C$$
is a strong monoidal functor, with the monoidal structure 
$$(X_1\otimes Y_1)\otimes(X_2\otimes Y_2)\to (X_1\otimes X_2)\otimes(Y_1\otimes Y_2)\quad (\forall X_1,X_2,Y_1,Y_2\in\cat C)$$
given by the parenthesized braid
\begin{equation}\label{eq:J}
\begin{tikzpicture}[baseline=1cm]
\coordinate (diff) at (0.65,0);
\coordinate (dy) at (0,0.5);
\node(X1) at (0,0) {$(X_1$};
\node(Y1) at ($(X1)+(diff)$) {$Y_1)$};
\node(Z1) at (2,0) {$(X_2$};
\node(W1) at ($(Z1)+(diff)$) {$Y_2)$};
\node(X2) at (0,2) {$(X_1$};
\node(Z2) at ($(X2)+(diff)$) {$X_2)$};
\node(Y2) at (2,2) {$(Y_1$};
\node(W2) at ($(Y2)+(diff)$) {$Y_2)$};
\draw(X1)--(X2);
\draw(W1)--(W2);
\draw (Z1)..controls +(0,1) and +(0,-1)..(Z2);
\draw[line width=1ex,white] (Y1)..controls +(0,1) and +(0,-1)..(Y2);
\draw(Y1)..controls +(0,1) and +(0,-1)..(Y2);
\end{tikzpicture}
\end{equation}
\end{thm}

\begin{proof}
Recall that a strong monoidal structure on a functor $F$ between two monoidal categories is a natural isomorphism $F(X\otimes Y)\to F(X)\otimes F(Y)$ such that the diagram
\begin{equation}\label{eq:JPhi}
\begin{tikzcd}
F((X\otimes Y)\otimes Z)\arrow{r}\arrow{d}& F(X\otimes Y)\otimes F(Z)\arrow{r} &(F(X)\otimes F(Y))\otimes F(Z)\arrow{d}\\
F(X\otimes (Y\otimes Z))\arrow{r}& F(X)\otimes F( Y\otimes Z)\arrow{r} &F(X)\otimes (F(Y)\otimes F(Z))
\end{tikzcd}
\end{equation}
commutes. In our case ($F=\otimes$) both ways of composing the morphisms from top left to the bottom right correspond to the same parenthesized braid, where all the strands moving to the right are above the strands moving to the left.

\end{proof}

\begin{cor}
If $A,B\in\cat C$ are monoids (and thus $(A,B)$ is a monoid in $\cat C\times\cat C$) then
$A\otimes B\in\cat C$ is also a monoid, with the product

$$
\begin{tikzpicture}[baseline=1cm]
\coordinate (diff) at (0.5,0);
\coordinate (dy) at (0,0.5);
\node(A1) at (0,0) {$A$};
\node(B1) at ($(A1)+(diff)$) {$B$};
\node(A2) at (2,0) {$A$};
\node(B2) at ($(A2)+(diff)$) {$B$};
\node(A3) at ($(A1)+(0,2.5)+0.5*(diff)$) {$A$};
\node(B3) at ($(A2)+(0,2.5)+0.5*(diff)$) {$B$};
\coordinate(A) at ($(A3)-(dy)$);
\coordinate(B) at ($(B3)-(dy)$);
\draw (A2)..controls +(0,1) and +(0,-1)..(A);
\draw[line width=1ex,white] (B1)..controls +(0,1) and +(0,-1)..(B);
\draw (B1)..controls +(0,1) and +(0,-1)..(B);
\draw (A1)..controls +(0,1) and +(0,-1)..(A);
\draw (B2)..controls +(0,1) and +(0,-1)..(B);
\draw (B)--(B3);
\draw(A)--(A3);
\end{tikzpicture}
$$
\end{cor}

In the case of $\cat C=U\g\Mod^\Phi$ there exists a universal element
$J\in(U\g)^{\otimes4}[\![\hbar]\!]$ such that the morphism \eqref{eq:J} is equal to
$$(1\otimes s_{Y_1,X_2}\otimes1)\circ(J\cdot),$$
where $s_{Y_1,X_2}:Y_1\otimes X_2\to X_2\otimes Y_1$ is the symmetry morphism. The property of $J$ given by 
Theorem \ref{thm:qfusion} also implies the following (slightly stronger) result.

\begin{thm}\label{thm:qfusion2}
Let $\g$ and $\h$ be Lie algebras with chosen elements $t_\g\in(S^2\g)^\g$, $t_\h\in(S^2\h)^\h$. Let
$$F:U(\g\oplus\g\oplus\h)\Mod^\Phi\to U(\g\oplus\h)\Mod^\Phi,\ F(X)=X$$
be the functor coming from the morphism of Lie algebras
$$\g\oplus\h\to\g\oplus\g\oplus\h,\ (u,v)\mapsto(u,u,v).$$
Then $F$ is a monoidal functor where the coherence map is the natural transformation defined by the action of
$J\in(U\g)^{\otimes4}[\![\hbar]\!]=\bigl(U(\g\oplus\g)\bigr)^{\otimes2}[\![\hbar]\!]$. 
\end{thm}
\begin{proof}
To simplify the notation, we assume $\h$ is trivial. The key thing to notice is that for $\cat C=U\g\Mod^\Phi$, the coherence map described in Theorem~\ref{thm:qfusion} is given by the universal element $J\in(U\g)^{\otimes4}[\![\hbar]\!]$, which in turn, satisfies a universal intertwining equation. Specifically, the element $J$ satisfies the relation
$$\Phi^{1,3,5}\Phi^{2,4,6}J^{13,24,5,6}J^{1,2,3,4}=J^{1,2,35,46}J^{3,4,5,6}\Phi^{12,34,56}\in (U\g)^{\otimes 6}[\![\hbar]\!],$$
as both sides of this identity come from the same parenthesized braid with 6 strands: the braid corresponding to either chain of morphisms from top-left to bottom right of \eqref{eq:JPhi}.
 This equation is precisely the identity which the coherence map is required to satisfy for Theorem~\ref{thm:qfusion2}.
\end{proof}

If $A\in U(\g\oplus\g\oplus\h)\Mod^\Phi$ is an associative algebra, we shall call the associative algebra $F(A)\in U(\g\oplus\h)\Mod^\Phi$ a \emph{(quantum) fusion} of $A$. Notice that $F$ is not a braided monoidal functor, hence $F(A)$ can be non-commutative even if $A$ is commutative. 

If $m:A\otimes A\to A$ is the original product in $A$ and $m'$ the fused product then
\begin{subequations}
\begin{equation}
m'=m\circ(J\cdot).
\end{equation}
Since
$$J=1+\frac{\hbar}{2}t^{2,3}+O(\hbar^2),$$
we get
\begin{equation}\label{eq:fus_hbar}
m'=m+\frac{\hbar}{2}m\circ(t^{2,3}\cdot)+O(\hbar^2).
\end{equation}
\end{subequations}

\section{Quasi-Poisson algebras}\label{sec:QPAlg}
Let $\g$ and $t\in(S^2\g)^\g$ be as above. Let $\phi\in\bigwedge^3\g\subset\g^{\otimes 3}$ be defined by
$$\phi=\frac{1}{4}[t^{1,2},t^{2,3}],$$
i.e.
$$\phi(\alpha,\beta,\gamma)=-\frac{1}{4}\la [t^\sharp\alpha,t^\sharp\beta],\gamma\ra\quad \forall\alpha,\beta,\gamma\in\g^*,$$
where $t^\sharp:\g^*\to\g$ is given by contraction with $t$.
Every associator $\Phi$ satisfies
\begin{equation}\label{eq:Phi_phi}
\Phi(\hbar t^{1,2},\hbar t^{2,3})=1+\frac{\hbar^2}{6}\phi+O(\hbar^3).
\end{equation}

\begin{defn}[\cite{Alekseev00}]
A \emph{$\g$-quasi-Poisson algebra} is a $\g$-module $A$ with a $\g$-invariant commutative associative product, and with a $\g$-invariant skew-symmetric bilinear map
$$\{\,,\,\}:A\times A\to A,$$
which is a derivation in both components, such that
\begin{equation}\label{eq:qPalg}
\{a_1,\{a_2,a_3\}\}+c.p.=-m^{(3)}(\phi\cdot(a_1\otimes a_2\otimes a_3)),
\end{equation}
for all $a_1,a_2,a_3\in A$,
where $m^{(3)}:A^{\otimes 3}\to A$ is the product.
\end{defn}

The following proposition is from \cite{Enriquez:2003tw}; for completeness, we include a proof.
\begin{prop}
Let $A$ be an associative algebra in $U\g\Mod^\Phi$ with product $m_\hbar:A\otimes A\to A$. Suppose that the reduction $m_0$ of $m_\hbar$ modulo $\hbar$ is a commutative product. Let
$$\{\,,\,\}:A\times A\to A$$
be defined by
$$\{a,b\}=m_\hbar(a\otimes b-b\otimes a)/\hbar\mod\hbar.$$
Then $(A,m_0,\{\,,\,\})$ is a $\g$-quasi-Poisson algebra.
\end{prop}
\begin{proof}
The associativity of $m_\hbar$ means
$$m_\hbar\circ(1\otimes m_\hbar)\circ \Phi_A=m_\hbar\circ(m_\hbar\otimes 1),$$
where
$$\Phi_A:A^{\otimes 3}\to A^{\otimes 3}$$
is the action of $\Phi(\hbar t^{1,2},\hbar t^{2,3})$ on $A^{\otimes 3}$.

Let $[a,b]:=m_\hbar(a\otimes b-b\otimes a)$. 
Using \eqref{eq:Phi_phi}
we get
\begin{multline*}
[a_1,[a_2,a_3]]+c.p.=\\
=\bigl(m_\hbar\circ(1\otimes m_\hbar)-m_\hbar\circ(m_\hbar\otimes 1)\bigr)\bigl(\sum_{\sigma\in S_3}\on{sgn}\sigma\, a_{\sigma(1)}\otimes a_{\sigma(2)}\otimes a_{\sigma(3)}\bigr)\\
= \bigl(m_\hbar\circ(1\otimes m_\hbar)\circ(1-\Phi_A)\bigr)\bigl(\sum_{\sigma\in S_3}\on{sgn}\sigma\, a_{\sigma(1)}\otimes a_{\sigma(2)}\otimes a_{\sigma(3)}\bigr)\\
=-\hbar^2(m_0\circ(1\otimes m_0))(\phi\cdot(a_1\otimes a_2\otimes a_3))+O(\hbar^3),
\end{multline*}
as we wanted to prove.

\end{proof}

The constructions of associative algebras in $U\g\Mod^\Phi$ have straightforward analogs for $\g$-quasi-Poisson algebras:
\begin{itemize}
\item 
If $A$ is a commutative associative algebra in $U\g\Mod$ such $m\circ(t^{1,2}\cdot)=0$ then $A$, with $\{,\}=0$, is a $\g$-quasi-Poisson algebra. We shall call these algebras \emph{quasi-Poisson-commutative}. This is true, in particular, for $A=C^\infty(M)$, where $M$ is a $\g$-manifold such that the stabilizers are coisotropic Lie subalgebras of $\g$ \cite{LiBland:2010wi}.

\item If $A$ is a $\g\oplus\h$-quasi-Poisson algebra and $\mf c\subset\g$ is a coisotropic Lie subalgebra then the space of $\mf c$-invariants $A^{\mf c}$ is an $\h$-quasi-Poisson algebra \cite{LiBland:2012vo}.

\item If $A$ is a $\g\oplus\g\oplus\h$-quasi-Poisson algebra then the morphism of Lie algebras
$$\g\oplus\h\to\g\oplus\g\oplus\h,\ (u,v)\mapsto(u,u,v)$$
makes $A$ to a $\g\oplus\h$-quasi-Poisson algebra, with the new bracket (cf.\ \eqref{eq:fus_hbar})
$$\{a,b\}'=\{a,b\}+\half m(t^{2,3}\cdot (a\otimes b-b\otimes a)),$$
where $m:A\otimes A\to A$ is the product. The result is called a \emph{(quasi-Poisson) fusion} of $A$ (see \cite{Alekseev00}).

\item If $A$ is a $\g\oplus\h$-quasi-Poisson algebra then it is, with the same bracket, also a $\bar\g\oplus\h$-quasi-Poisson algebra. The same is true for algebras in $U(\g\oplus\h)\Mod^\Phi$ and $U(\bar\g\oplus\h)\Mod^\Phi$ only if the associator $\Phi$ is even, i.e.\ if 
$$\Phi(-x,-y)=\Phi(x,y),$$
since then the the monoidal structures on $U(\g\oplus\h)\Mod^\Phi$ and $U(\bar\g\oplus\h)\Mod^\Phi$ are the same (with different braidings).
\end{itemize}

\begin{defn}[\cite{Enriquez:2003tw}]
A \emph{deformation quantization} of a $\g$-quasi-Poisson algebra $$(A,m_0,\{,\})$$ is a series $m_\hbar=\sum_{i=0}^\infty \hbar^i m_{(i)}$ making $A$ an associative algebra in $U\g\Mod^\Phi$, such that 
$$m_{(0)}=m_0,\quad \{a,b\}=m_{(1)}(a\otimes b-b\otimes a),$$
 and such that each $m_{(i)}$ is a bidifferential operator (with respect to $m_0)$.
\end{defn}

The following theorem gives us a deformation quantization of any quasi-Poisson algebra which is built  out of a quasi-Poisson-commutative algebra by repeated fusion and reduction.

\begin{thm}\label{thm:quantization}
\begin{enumerate}
\item If $(A,m_0,\{,\}=0)$ is a $\g$-quasi-Poisson-commutative algebra then $A$, with $m_\hbar=m_0$, is its deformation quantization.
\item If $(A,m_\hbar)$ is a deformation quantization of a $\g\oplus\g\oplus\h$-quasi-Poisson algebra $(A,m_0,\{,\})$ then the quantum fusion of $(A,m_\hbar)$ is a deformation quantization of the quasi-Poisson fusion of $(A,m_0,\{,\})$.
\item If $(A,m_\hbar)$ is a deformation quantization of a $\g\oplus\h$-quasi-Poisson algebra $(A,m_0,\{,\})$ and if $\mf c\subset\g$ is a coisotropic Lie subalgebra then $A^{\mf c}$ is a deformation quantization of the $\h$-quasi-Poisson algebra $A^{\mf c}$.
\end{enumerate}
\end{thm}
\begin{proof}
Part 1 is Proposition \ref{prop:comm}, part 2 follows from Equation \eqref{eq:fus_hbar}, and part 3 is obvious.
\end{proof}

\section{Quasi-Poisson structures on moduli spaces of flat connections}\label{sec:moduli}

If $M$ is a manifold with a $\g$-quasi-Poisson structure on $C^\infty(M)$ then $M$ is called a \emph{$\g$-quasi-Poisson manifold}. Equivalently, $M$ is a manifold endowed with an action of $\g$ and with a $\g$-invariant bivector field $\pi$ such that
$$[\pi,\pi]/2=-\phi_M,$$
where $\phi_M$ is the image of $\phi\in\bigwedge^3\g$ under the action map $\g\to \mf{X}(M)$.

We shall say that $M$ is \emph{quasi-Poisson-commutative} if $C^\infty(M)$ is quasi-Poisson-commutative, i.e.\ if $\pi=0$ and the action of $\g$ has coisotropic stabilizers.

If $M$ is a $\g$-quasi-Poisson manifold and $N$ a $\h$-quasi-Poisson manifold then $M\times N$ is $\g\oplus\h$-quasi-Poisson, with $\pi_{M\times N}=\pi_M+\pi_N$. If $M$ is $\g$-quasi-Poisson and $M'\to M$ is a local diffeomorphism (an \'etale map) then $M'$ is also $\g$-quasi-Poisson.

The most important examples of quasi-Poisson manifolds arise as moduli spaces of flat connections on a surface \cite{Alekseev00,Alekseev97}. Our presentation follows \cite{LiBland:2012vo}.

Let $\Sigma$ be a compact oriented surface with boundary and $V\subset\partial\Sigma$ a finite set of marked points. We suppose that $V$ meets every component of $\Sigma$ ($\Sigma$ doesn't have to be connected).

Let $\Pi_1(\Sigma,V)$ denote the fundamental groupoid of $\Sigma$ with the base set $V$. Let
$$M_{\Sigma,V}(G)=\Hom(\Pi_1(\Sigma,V),G).$$
By a \emph{skeleton} of $(\Sigma,V)$, we mean an embedded oriented graph $\Gamma\subset\Sigma$ with the vertex set $V$, such that there is a deformation retraction of $\Sigma$ to $\Gamma$.
Given a skeleton, we get a bijection
$$M_{\Sigma,V}(G)\cong G^{E},$$
where $E$ is the set of edges of $\Gamma$. In this way $M_{\Sigma,V}(G)$ becomes a manifold (the manifold structure is independent of the choice of $\Gamma$).

There is a natural action of $G^V$ on $M_{\Sigma,V}(G)$, namely
$$(g\cdot\mu)(\gamma)=g^{\phantom{-1}}_{\on{in}(\gamma)}\mu(\gamma)g^{-1}_{\on{out}(\gamma)}\quad
(\mu:\Pi_1(\Sigma,V)\to G,\gamma\in\Pi_1(\Sigma,V),g\in G^V).$$
\begin{center}
\begingroup%
  \makeatletter%
  \providecommand\color[2][]{%
    \errmessage{(Inkscape) Color is used for the text in Inkscape, but the package 'color.sty' is not loaded}%
    \renewcommand\color[2][]{}%
  }%
  \providecommand\transparent[1]{%
    \errmessage{(Inkscape) Transparency is used (non-zero) for the text in Inkscape, but the package 'transparent.sty' is not loaded}%
    \renewcommand\transparent[1]{}%
  }%
  \providecommand\rotatebox[2]{#2}%
  \ifx\svgwidth\undefined%
    \setlength{\unitlength}{265.66932656bp}%
    \ifx\svgscale\undefined%
      \relax%
    \else%
      \setlength{\unitlength}{\unitlength * \real{\svgscale}}%
    \fi%
  \else%
    \setlength{\unitlength}{\svgwidth}%
  \fi%
  \global\let\svgwidth\undefined%
  \global\let\svgscale\undefined%
  \makeatother%
  \begin{picture}(1,0.41648039)%
    \put(0,0){\includegraphics[width=\unitlength]{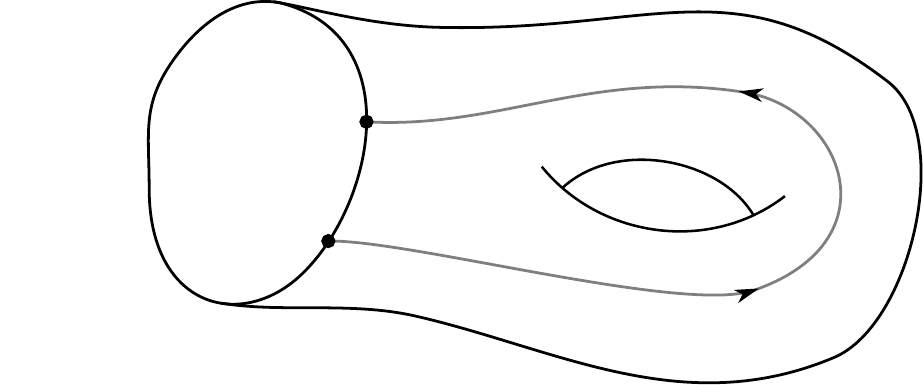}}%
    \put(0.33992195,0.15514438){\color[rgb]{0,0,0}\makebox(0,0)[rb]{\smash{$\on{out}(\gamma)$}}}%
    \put(0.37992415,0.28448479){\color[rgb]{0,0,0}\makebox(0,0)[rb]{\smash{$\on{in}(\gamma)$}}}%
    \put(0.14384364,0.20168041){\color[rgb]{0,0,0}\makebox(0,0)[rb]{\smash{$\Sigma=$}}}%
    \put(0.7421688,0.33390241){\color[rgb]{0,0,0}\makebox(0,0)[b]{\smash{$\gamma$}}}%
  \end{picture}%
\endgroup%

\end{center}

Recall from Example~\ref{ex:CG} that $\g\oplus\bar\g$ acts on $G$ with coisotropic stabilizers. Thus $G$, with $\pi=0$, is $\g\oplus\bar\g$-quasi-Poisson, and hence also $\g\oplus\g$-quasi-Poisson. Notice that it is quasi-Poisson-commutative in the first case, but not in the second case.

\begin{thm}[\cite{LiBland:2012vo}]\label{thm:ModBiv}
There is a natural bivector field $\pi_{\Sigma,V}$ on $M_{\Sigma,V}(G)$ which, together with the action of $\g^V$, makes $M_{\Sigma,V}(G)$ to a $\g^V$-quasi-Poisson manifold. It is specified uniquely by the following properties.
\begin{itemize}
\item If $\Sigma$ is a disk and $V$ consists of 2 points, so that $M_{\Sigma,V}(G)=G$, then $\pi_{\Sigma,V}=0$.

\item If $\Sigma$ is the disjoint union of $\Sigma_1$ and $\Sigma_2$, so that $M_{\Sigma,V}(G)=M_{\Sigma_1,V_1}(G)\times M_{\Sigma_2,V_2}(G)$, then $\pi_{\Sigma,V}=\pi_{\Sigma_1,V_1}+\pi_{\Sigma_2,V_2}$.

\item If $(\Sigma^*,V^*)$ is obtained from $(\Sigma, V)$ by a ``corner connected sum'' at two points $P,Q\in V$,  as in the picture, then $M_{\Sigma^*,V^*}(G)$ is obtained by fusing the $\g^{V}$-quasi-Poisson structure on $M_{\Sigma,V}(G)$ along the
the $P^{th}$ and $Q^{th}$ factors of $\g^{V}$. 
\end{itemize}
\begin{center}
\begingroup%
  \makeatletter%
  \providecommand\color[2][]{%
    \errmessage{(Inkscape) Color is used for the text in Inkscape, but the package 'color.sty' is not loaded}%
    \renewcommand\color[2][]{}%
  }%
  \providecommand\transparent[1]{%
    \errmessage{(Inkscape) Transparency is used (non-zero) for the text in Inkscape, but the package 'transparent.sty' is not loaded}%
    \renewcommand\transparent[1]{}%
  }%
  \providecommand\rotatebox[2]{#2}%
  \ifx\svgwidth\undefined%
    \setlength{\unitlength}{350.1152832bp}%
    \ifx\svgscale\undefined%
      \relax%
    \else%
      \setlength{\unitlength}{\unitlength * \real{\svgscale}}%
    \fi%
  \else%
    \setlength{\unitlength}{\svgwidth}%
  \fi%
  \global\let\svgwidth\undefined%
  \global\let\svgscale\undefined%
  \makeatother%
  \begin{picture}(1,0.25565742)%
    \put(0,0){\includegraphics[width=\unitlength]{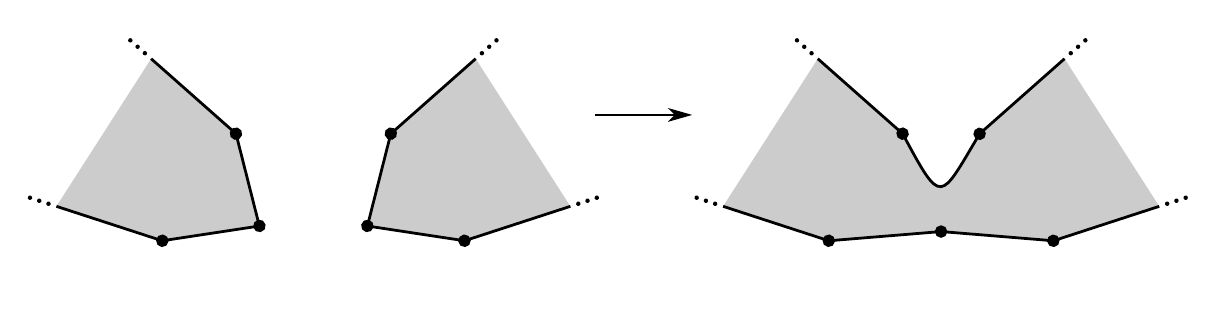}}%
    \put(0.13191874,0.0262079){\color[rgb]{0,0,0}\makebox(0,0)[b]{\smash{$A$}}}%
    \put(0.19412736,0.15372789){\color[rgb]{0,0,0}\makebox(0,0)[lb]{\smash{$B$}}}%
    \put(0.37928605,0.0262079){\color[rgb]{0,0,0}\makebox(0,0)[b]{\smash{$C$}}}%
    \put(0.32314822,0.15372789){\color[rgb]{0,0,0}\makebox(0,0)[rb]{\smash{$D$}}}%
    \put(0.21792138,0.07581959){\color[rgb]{0,0,0}\makebox(0,0)[lb]{\smash{$P$}}}%
    \put(0.29781964,0.07581959){\color[rgb]{0,0,0}\makebox(0,0)[rb]{\smash{$Q$}}}%
    \put(0.25486881,0.21544745){\color[rgb]{0,0,0}\makebox(0,0)[b]{\smash{$\Sigma$}}}%
    \put(0.68030952,0.0262079){\color[rgb]{0,0,0}\makebox(0,0)[b]{\smash{$A$}}}%
    \put(0.74251818,0.15372789){\color[rgb]{0,0,0}\makebox(0,0)[lb]{\smash{$B$}}}%
    \put(0.86369792,0.0262079){\color[rgb]{0,0,0}\makebox(0,0)[b]{\smash{$C$}}}%
    \put(0.80756009,0.15372789){\color[rgb]{0,0,0}\makebox(0,0)[rb]{\smash{$D$}}}%
    \put(0.77545201,0.07581959){\color[rgb]{0,0,0}\makebox(0,0)[b]{\smash{$P^*=Q^*$}}}%
    \put(0.77127014,0.21544745){\color[rgb]{0,0,0}\makebox(0,0)[b]{\smash{$\Sigma^*$}}}%
  \end{picture}%
\endgroup%

\end{center}

If $P\in V$ is such that the set $V-\{P\}$ meets every component of $\Sigma$ then $M_{\Sigma,V-\{P\}}(G)$ is the reduction of $M_{\Sigma,V}(G)$ by $\g$ acting at $P$. 
\end{thm}

In particular, a deformation retraction of $\Sigma$ onto the skeleton $\Gamma$ makes $\Sigma$ to a corner connected sum of a family of disks with 2 marked points each (one disk for each edge $e\in E$; the disk is the part of $\Sigma$ that retracts onto $e$, as in Figure~\ref{fig:fusedSurf}), hence the $\g^V$-quasi-Poisson manifold $M_{\Sigma,V}(G)=G^E$ is obtained from the $\g^{2E}$-quasi-Poisson manifold $G^E$ (with $\{\cdot,\cdot\}=0$) by repeated fusion.

For the purpose of quantization it is more convenient to produce the quasi-Poisson manifold $M_{\Sigma,V}(G)$ out of a quasi-Poisson-commutative manifold. Let us split $V$ to two disjoint subsets $V=V_+\sqcup V_-$. The manifold $M_{\Sigma,V}(G)$ is $\g^{V_+}\oplus\bar\g^{V_-}$-quasi-Poisson, with the same bivector field $\pi_{\Sigma,V}$. If both $V_+$ and $V_-$ meet every component of $\Sigma$ then there exists a skeleton $\Gamma$ which is bipartite with respect to $V_+$ and $V_-$, with edges oriented from $V_-$ to $V_+$. In this case we start with $G^E$ as a $\g^E\oplus\bar\g^E$-quasi-Poisson-commutative manifold, and obtain the $\g^{V_+}\oplus\bar\g^{V_-}$-quasi-Poisson structure on $M_{\Sigma,V}(G)$ by fusion.

If we don't want to split $V$ (or if it doesn't admit a convenient splitting), we can set $V_+=V$, temporarily add new points to $V_-$, and finally reduce the result by $\bar\g^{V_-}$.

\begin{figure}\label{fig}
\begin{center}
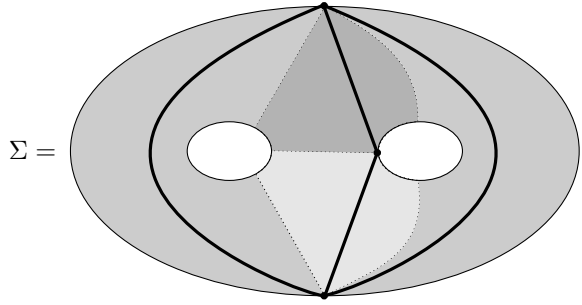
\end{center}
\caption{\label{fig:fusedSurf} The surface $\Sigma$ retracts onto the embedded skeletal graph. The pre-images (under the retract) of the graph's various edges are shaded in different tones; each such pre-image is diffeomorphic to a disk. In particular the retract yields an explicit decomposition of the surface as a corner connected sum of a family of disks each of which has two marked points.}
\end{figure}

\section{Deformation quantization of Poisson manifolds related to Poisson-Lie groups}

Let $\Sigma$ and $V\subset\partial\Sigma$ be as above. For every marked point $P\in V$ we choose a coisotropic Lie subalgebra $\mf c_P\subset\g$. Let $C_P\subset G$ be the corresponding connected Lie group. If the action of 
$$C=\prod_P C_P\subset G^V$$
on $M_{\Sigma,V}(G)$ is free and proper then the quotient
\begin{equation}\label{eq:Mred}
M_{\Sigma,V}(G)/C
\end{equation}
is a Poisson manifold. To deal with cases where the action is not free and proper we can sometimes replace $M_{\Sigma,V}(G)$ with another manifold $M'$ which is equipped with a free and proper action of $C$ along with a local $C$-equivariant diffeomorphism
$M'\to M_{\Sigma,V}(G)$; the lift of $\pi_{\Sigma,V}$ to $M'$ descends to a Poisson structure on $M'/C$.

Since the Poisson manifold \eqref{eq:Mred} can be obtained from commutative quasi-Poisson manifolds by fusion and reduction,  we readily obtain its deformation quantization by repeating the same steps in the world of associative algebras in  Drinfeľd categories (see Theorem \ref{thm:quantization}).

\begin{rem}
Our quantization of the Poisson manifold $M_{\Sigma,V}(G)/C$ depends on the choice of a skeleton of $(\Sigma, V)$. Let us make a conjecture about this dependence. If $\Gamma$ is a skeleton, let $A_\Gamma$ denote the corresponding quantized algebra of functions on  $M_{\Sigma,V}(G)/C$. Then there are isomorphisms $\phi_{\Gamma',\Gamma}:A_\Gamma\to A_{\Gamma'}$, natural up to inner automorphisms, and invertible elements $c_{\Gamma'',\Gamma',\Gamma}\in A_\Gamma$ such that
$$\phi_{\Gamma'',\Gamma'}\circ\phi_{\Gamma',\Gamma}=\phi_{\Gamma'',\Gamma}\circ\on{Ad}_{c_{\Gamma'',\Gamma',\Gamma}},$$
satisfying the cocycle condition
$$c_{\Gamma_3,\Gamma_2,\Gamma_0}\,c_{\Gamma_2,\Gamma_1,\Gamma_0}=c_{\Gamma_3,\Gamma_1,\Gamma_0}\,\phi_{\Gamma_1,\Gamma_0}^{-1}(c_{\Gamma_3,\Gamma_2,\Gamma_1}).$$
Equivalently, there is a natural linear category whose objects are the skeletons of $(\Sigma,V)$, such that any two objects are isomorphic (i.e.\ an algebroid in the sense of Kontsevich), and $\on{End}(\Gamma)=A_\Gamma$; if we choose isomorphisms $\psi_{\Gamma',\Gamma}:\Gamma\to\Gamma'$ for every $\Gamma,\Gamma'$ then we can set $\phi_{\Gamma',\Gamma}=\on{Ad}_{\psi_{\Gamma',\Gamma}}$ and $c_{\Gamma'',\Gamma',\Gamma}=\psi_{\Gamma'',\Gamma}^{-1}\circ\psi_{\Gamma'',\Gamma'}\circ\psi_{\Gamma',\Gamma}$.

This statement was proven in \cite{ko} by Kontsevich for his quantization of Poisson manifolds, where in place of $\Gamma$'s he used affine connections on the manifold; it is thus natural to expect it also in our case.
\end{rem}

We shall suppose in the following examples that $t\in S^2\g$ is non-degenerate, unless explicitly stated otherwise. The inverse of $t$ is thus a non-degenerate invariant symmetric bilinear pairing $\langle\cdot,\cdot\rangle$ on $\g$. These examples are taken from \cite{LiBland:2012vo}. We first consider several Poisson manifolds related to Poisson-Lie groups, and finally we discuss the general case.

\subsection{Poisson-Lie groups}
Let $\h,\h^*\subset\g$ be Lie subalgebras which are Lagrangian with respect to $\langle\cdot,\cdot\rangle$, such that $\g=\h\oplus\h^*$ as a vector space. In other words, $\h,\h^*\subset\g$ is a Manin triple.

Let us first suppose that the map
$$H\times H^*\to G$$
given by the product in $G$ is a diffeomorphism (it is always a local diffeomorphism).
Let us consider the moduli space for the surface
$$
\begin{tikzpicture}[baseline=-1cm]
\coordinate[label=left:{$\h^*$}] (A) at (0,0) ;
\coordinate[label=right:{$\h$}] (B) at (2,0) ;
\coordinate[label=below:{$\h^*$}] (C) at (1,-1.7);
\filldraw[fill=white!90!black] (A)--(B)--(C)--cycle;
\fill (A) circle (0.05) (B) circle (0.05) (C) circle (0.05);
\end{tikzpicture}
$$
For each marked point $P\in V$ we specify the corresponding Lie subalgebra $\mf c_P\subset\g$  on the figure. If no subalgebra is specified then we do not reduce at that marked point (and the result is quasi-Poisson).

Explicitly, we construct this moduli space by the following sequence of fusions and reductions:
$$
\begin{tikzpicture}[baseline=-1cm]
\coordinate[label=left:{$\h^*$}] (A0) at (-4,0) ;
\coordinate[label=right:{$\h$}] (B0) at (-1.5,0) ;
\coordinate (C0) at (-3,-1.7);
\coordinate (C0') at (-2.5,-1.7);
\filldraw[fill=white!90!black] (A0) .. controls (-3,0) and (-3,-.7) .. (C0)--cycle;
\filldraw[fill=white!90!black] (B0) .. controls (-2.5,0) and (-2.5,-.7) ..(C0')--cycle;
\fill (A0) circle (0.05) (B0) circle (0.05) (C0) circle (0.05) (C0') circle (0.05);
\draw[->,decorate,
     decoration={snake,amplitude=.4mm,segment length=2mm,post length=1mm}] (-1.3,-1) -- (-.2,-1);

\coordinate[label=left:{$\h^*$}] (A) at (0,0) ;
\coordinate[label=right:{$\h$}] (B) at (2,0) ;
\coordinate (C) at (1,-1.7);
\filldraw[fill=white!90!black] (A)--(B)--(C)--cycle;
\fill (A) circle (0.05) (B) circle (0.05) (C) circle (0.05);
\path[dotted,-] (C) edge (1,0);
\draw[->,decorate,
     decoration={snake,amplitude=.4mm,segment length=2mm,post length=1mm}] (2.2,-1) -- (3.3,-1);

\coordinate[label=left:{$\h^*$}] (A1) at (3.5,0) ;
\coordinate[label=right:{$\h$}] (B1) at (5.5,0) ;
\coordinate[label=below:{$\h^*$}] (C1) at (4.5,-1.7);
\filldraw[fill=white!90!black] (A1)--(B1)--(C1)--cycle;
\fill (A1) circle (0.05) (B1) circle (0.05) (C1) circle (0.05);
\end{tikzpicture}
$$
In the left most picture, we start with the $\g$-quasi-Poisson-commutative manifolds $G/H^*$ and $G/H$, and fuse them to form the triangle pictured in the center. Finally, we reduce by $H^*\subset G$ to make
$$H\cong (G/H^*\times G/H)/H^*$$
 a Poisson manifold.
We  identify it with $H$ by setting the holonomy of the top edge to be 1 and by demanding that the holonomy of the left edge is in $H$. The resulting Poisson structure on $H$ is the Poisson-Lie structure given by the Manin triple (see \cite{LiBland:2012vo} for more details).

Repeating the same steps in the Drinfeľd category, we endow
$$C^\infty(H)[\![\hbar]\!]\cong C^\infty(G/H^*\times G/H)^{\h^*}[\![\hbar]\!]$$
with an associative multiplication, i.e.\ we construct a star product quantizing the Poisson structure on $H$.

We can eliminate the assumption that $H\times H^*\to G$ is a diffeomorphism by replacing $G/H^*$ with $H$ and $G/H$ with $H^*$ (since they are locally diffeomorphic).

An inspection shows that the star product at $1\in H$ coincides with the deformed coproduct on $U\h$ constructed by Etingof and Kazhdan \cite{Etingof:1996bc}. We don't prove this statement here for lack of space; see the end of Section \ref{subs:double} for the basic idea.

\begin{rem}
There is another way how to get the Poisson-Lie group $H$ via fusion and reduction (corresponding to a different choice of skeleton):
$$
\begin{tikzpicture}[baseline=-1cm]
\coordinate[label=left:{$\h^*$}] (A0) at (-4,0) ;
\coordinate[label=right:{$\h^*$}] (B0) at (-1.5,0) ;
\coordinate (C0) at (-3,-1.7);
\coordinate (C0') at (-2.5,-1.7);
\filldraw[fill=white!90!black] (A0) .. controls (-3,0) and (-3,-.7) .. (C0)--cycle;
\filldraw[fill=white!90!black] (B0) .. controls (-2.5,0) and (-2.5,-.7) ..(C0')--cycle;
\fill (A0) circle (0.05) (B0) circle (0.05) (C0) circle (0.05) (C0') circle (0.05);
\draw[->,decorate,
     decoration={snake,amplitude=.4mm,segment length=2mm,post length=1mm}] (-1.3,-1) -- (-.2,-1);

\coordinate[label=left:{$\h^*$}] (A) at (0,0) ;
\coordinate[label=right:{$\h^*$}] (B) at (2,0) ;
\coordinate (C) at (1,-1.7);
\filldraw[fill=white!90!black] (A)--(B)--(C)--cycle;
\fill (A) circle (0.05) (B) circle (0.05) (C) circle (0.05);
\path[dotted,-] (C) edge (1,0);
\draw[->,decorate,
     decoration={snake,amplitude=.4mm,segment length=2mm,post length=1mm}] (2.2,-1) -- (3.3,-1);

\coordinate[label=left:{$\h^*$}] (A1) at (3.5,0) ;
\coordinate[label=right:{$\h^*$}] (B1) at (5.5,0) ;
\coordinate[label=below:{$\h$}] (C1) at (4.5,-1.7);
\filldraw[fill=white!90!black] (A1)--(B1)--(C1)--cycle;
\fill (A1) circle (0.05) (B1) circle (0.05) (C1) circle (0.05);
\end{tikzpicture}
$$
The resulting star product on $H$ coincides with the quantization of Lie bialgebras described in \cite{se}.

\end{rem}

\subsection{Poisson torsors}

Let $\mf f\subset\g$ be another Lagrangian Lie subalgebra such that $\mf f\cap\h=0$. It is equivalent to (the graph of) a twist of the Lie bialgebra $\h^*$.

Let us consider the moduli space for the surface
$$
\begin{tikzpicture}[baseline=-1cm]
\coordinate[label=left:{$\h^*$}] (A) at (0,0) ;
\coordinate[label=right:{$\h$}] (B) at (2,0) ;
\coordinate[label=below:{$\mf f$}] (C) at (1,-1.7);
\filldraw[fill=white!90!black] (A)--(B)--(C)--cycle;
\fill (A) circle (0.05) (B) circle (0.05) (C) circle (0.05);
\end{tikzpicture}
$$
As above, the moduli space is constructed by fusing $G/H$ with $G/H^*$, and then reducing by $F\subseteq G$. Once again, we can identify this moduli space with $H$; the resulting Poisson structure on $H$ is the affine Poisson structure given by the twist $\mf f$.

Again (after replacing $G/H$ with $H^*$ and $G/H^*$ with $H$) following the analogous procedure of fusion followed by reduction in the Drinfeľd category yields an associative product on
$$C^\infty(H)[\![\hbar]\!]\cong C^\infty(H\times H^*)^{\mf f}[\![\hbar]\!],$$
which is a star product quantizing the affine Poisson structure on $H$.

\subsection{Drinfeľd double}\label{subs:double}

Let us consider the moduli space for the square
$$
\begin{tikzpicture}[baseline=0.8cm]
\coordinate[label=left:{$\h^*$}] (A) at (0,0) ;
\coordinate[label=right:{$\h$}] (B) at (2,0) ;
\coordinate[label=right:{$\h$}] (C) at (2,2);
\coordinate[label=left:{$\h^*$}] (D) at (0,2) ;
\filldraw[fill=white!90!black] (A)--(B)--(C)--(D)--cycle;
\fill (A) circle (0.05) (B) circle (0.05) (C) circle (0.05) (D) circle (0.05);
\end{tikzpicture}
$$
We can identify it with $G$ by demanding the holonomies along the horizontal edges to be 1 (the holonomies along the vertical edges are then arbitrary - but equal - elements of $G$). The group $G$ with this Poisson structure is the Drinfeľd double of $H$.

It is useful to construct this moduli space via the intermediate step
$$
\begin{tikzpicture}[baseline=0.8cm]
\coordinate[label=left:{$\h^*$}] (A) at (0,0) ;
\coordinate[label=right:{$\h$}] (B) at (2,0) ;
\coordinate[label=right:{$\h$}] (C) at (2,2);
\coordinate[label=left:{$\h^*$}] (D) at (0,2);
\coordinate(E) at (1,0.8);
\coordinate(F) at (1,1.2);
\filldraw[fill=white!90!black] (A)--(B)--(E)--cycle (C)--(D)--(F)--cycle;
\fill (A) circle (0.05) (B) circle (0.05) (C) circle (0.05) (D) circle (0.05) (E) circle (0.05) (F) circle (0.05);
\path[dotted,-] (E) edge (1,0);
\path[dotted,-] (F) edge (1,2);
\end{tikzpicture}
$$

The lower triangle corresponds to the fusion of $G/H\times G/H^*$, the upper triangle corresponds to the fusion of $\bar G/H^*\times \bar G/H$ (we consider it as a $\bar\g$-quasi-Poisson manifold for convenience). We take their product and reduce by the diagonal $G\subset G\times\bar G$.

Repeating these steps in Drinfeľd categories yields a deformation quantization of the double, which is again equal (at $1\in G$) to the quantization of the double given by Etingof and Kazhdan \cite{Etingof:1996bc}. We don't have the space to summarize the procedure of Etingof and Kazhdan here. Let us, however, recall that their main step is a definition of a coproduct $\Delta^\Phi$ on $U\g$, making it a coassociative coalgebra in the Drinfeľd category $U\g\Mod^\Phi$. The element $\Delta^\Phi1\in U\g\otimes U\g$ is then a twist turning the quasi-Hopf algebra $U\g$ (with the associator $\Phi$) to a Hopf algebra. The lower triangle on the figure above (the fusion of $G/H\times G/H^*$) gives us a star product on $G$ making $C^\infty(G)$ an associative algebra in $U\g\Mod^\Phi$. By construction, this star product is dual to the coproduct $\Delta^\Phi$ on $U\g$. From here it is not difficult to see that our quantization of the double coincides with the quantization of Etingof and Kazhdan. 

\subsection{Heisenberg double}
We can change the previous example slightly and consider the square
$$
\begin{tikzpicture}[baseline=0.8cm]
\coordinate[label=left:{$\h^*$}] (A) at (0,0) ;
\coordinate[label=right:{$\h$}] (B) at (2,0) ;
\coordinate[label=right:{$\h^*$}] (C) at (2,2);
\coordinate[label=left:{$\h$}] (D) at (0,2) ;
\filldraw[fill=white!90!black] (A)--(B)--(C)--(D)--cycle;
\fill (A) circle (0.05) (B) circle (0.05) (C) circle (0.05) (D) circle (0.05);
\end{tikzpicture}
$$
The moduli space can again be identified with
$G$. As a Poisson manifold, it is the so-called Heisenberg double of $H$; up to local diffeomorphism, it is also the Lu-Weinstein double symplectic groupoid. Repeating the analogous constructions in the Drinfeľd category yields quantizations of these spaces.

\subsection{Moduli spaces}
Let us now summarize how to quantize the moduli space \eqref{eq:Mred} in general (we don't suppose that $t$ is non-degenerate anymore). We represent $(\Sigma,V)$ as a corner-connected sum of disks with 2 marked points. By Theorem~\ref{thm:ModBiv}, the moduli space is formed by fusion and reduction from copies of the moduli space for the disk with 2 marked points. Hence, by Theorem~\ref{thm:quantization}, it is enough to quantize $M_{\Sigma,V}(G)\cong G$ in the special case when $(\Sigma,V)$ is a disk with 2 marked points.

The algebra $C^\infty(G)$ (with the original product) is associative (and commutative) in $U(\g\oplus\bar\g)\Mod^\Phi$. If the associator $\Phi$ is even then it is also associative in $U(\g\oplus\g)\Mod^\Phi$ (as the associativity constraints are the same). If not, we consider $C^\infty(G\times G)$ as a commutative algebra in $U(\g\oplus\bar\g\oplus\g\oplus\bar\g)\Mod^\Phi$, fuse it to form an algebra in $U(\g\oplus\g\oplus\bar\g)\Mod^\Phi$, and finally reduce by $\bar\g$ to make $C^\infty(G)$ (with a deformed product) into an associative algebra in $U(\g\oplus\g)\Mod^\Phi$. Pictorially, this construction is as follows:
$$
\begin{tikzpicture}[baseline=-1cm,scale=1]
\coordinate[label=left:{$+$}] (A00) at (-7.5,0) ;
\coordinate[label=right:{$+$}] (B00) at (-5,0) ;
\coordinate[label=below:{$-$}] (C00) at (-6.5,-1.7);
\coordinate[label=below:{$-$}] (C00') at (-6,-1.7);
\filldraw[fill=white!90!black] (A00) .. controls (-6.5,0) and (-6.5,-.7) .. (C00)--cycle;
\filldraw[fill=white!90!black] (B00) .. controls (-6,0) and (-6,-.7) .. (C00')--cycle;
\fill (A00) circle (0.05) (B00) circle (0.05) (C00) circle (0.05) (C00') circle (0.05);
\draw[->,decorate,
     decoration={snake,amplitude=.4mm,segment length=2mm,post length=1mm}] (-4.8,-1) -- (-3.7,-1);

\coordinate[label=left:{$+$}] (A0) at (-3.5,0) ;
\coordinate[label=right:{$+$}] (B0) at (-1.5,0) ;
\coordinate[label=below:{$-$}] (C0) at (-2.5,-1.7);
\filldraw[fill=white!90!black] (A0)--(B0)--(C0)--cycle;
\fill (A0) circle (0.05) (B0) circle (0.05) (C0) circle (0.05);
\path[dotted,-] (C0) edge (-2.5,0);
\draw[->,decorate,
     decoration={snake,amplitude=.4mm,segment length=2mm,post length=1mm}] (-1.3,-1) -- (-.2,-1);

\coordinate[label=left:{$+$}] (A1) at (0,-.5) ;
\coordinate[label=right:{$+$}] (B1) at (2,-.5) ;
\filldraw[fill=white!90!black] (A1) .. controls (.5,-1.5) and (1.5,-1.5) .. (B1)--cycle;
\fill (A1) circle (0.05) (B1) circle (0.05);
\end{tikzpicture}
$$
Here we have labeled the marked points at which $\g$ and $\bar\g$ act by $+$ and $-$ signs (respectively).

A more flexible way of dealing with non-even associators is to use a decomposition $V=V_+\sqcup V_-$ as at the end of Section \ref{sec:moduli}. In this case we just apply fusion to the commutative algebra $C^\infty(G^E)$ in $U(\g^E\oplus\bar\g^E)\Mod^\Phi$ (and, if we added extra marked points to $V_-$, reduce at those extra points to eliminate them).

Finally, to get a quantization of the Poisson manifold $M_{\Sigma,V}(G)/C$, we take the $\mf c$-invariants of the deformed algebra $C^\infty(M_{\Sigma,V}(G))$.


\begin{thebibliography}{1}

\bibitem{Alekseev00}
Anton~Yu Alekseev, Yvette Kosmann-Schwarzbach, and Eckhard Meinrenken.
\newblock {Quasi-Poisson manifolds}.
\newblock {\em Canadian Journal of Mathematics}, 54(1):3--29, 2002.

\bibitem{Alekseev97}
Anton~Yu Alekseev, Anton~Z. Malkin, and Eckhard Meinrenken.
\newblock {Lie group valued moment maps}.
\newblock {\em Journal of Differential Geometry}, 48(3):445--495, 1998.

\bibitem{Drinfeld:1989tu}
Vladimir~Gershonovich Drinfel'd.
\newblock {\em {Quasi-Hopf Algebras}}.
\newblock Algebra i Analiz, 1989.

\bibitem{Drinfeld:1990wg}
Vladimir~Gershonovich Drinfel'd.
\newblock {On quasitriangular quasi-Hopf algebras and on a group that is
  closely connected with $\rm Gal(\overline\bf Q/\bf Q)$}.
\newblock {\em Algebra i Analiz}, 2(4):149--181, 1990.

\bibitem{Enriquez:2003tw}
Benjamin Enriquez and Pavel Etingof.
\newblock {Quantization of Alekseev-Meinrenken dynamical $r$-matrices}.
\newblock In {\em Lie groups and symmetric spaces}, pages 81--98. Amer. Math.
  Soc., Providence, RI, 2003.

\bibitem{Etingof:1996bc}
Pavel Etingof and David Kazhdan.
\newblock {Quantization of Lie bialgebras. I}.
\newblock {\em Selecta Mathematica. New Series}, 2(1):1--41, 1996.

\bibitem{LiBland:2010wi}
David Li-Bland and Pavol {\v S}evera.
\newblock {Quasi-Hamiltonian groupoids and multiplicative Manin pairs}.
\newblock {\em International Mathematics Research Notices},
  2011(10):2295--2350, 2011.

\bibitem{LiBland:2012vo}
David Li-Bland and Pavol {\v S}evera.
\newblock {Moduli spaces for quilted surfaces and Poisson structures}.
\newblock December 2012.

\end{thebibliography}
\end{document}